\documentclass[11pt]{amsart}

\usepackage{amsmath}
\usepackage{amssymb}
\usepackage{amscd}
\usepackage{color}
\usepackage[all]{xy}

\begin{document}

\title{A Generalization of Reflexive Rings}

\author{M. B. Calci}
\address{Mete Burak Calci, Department of Mathematics, Ankara University,  Turkey}
\email{mburakcalci@gmail.com}

\author{H. Chen}
\address{Huanyin Chen, Department of Mathematics, Hangzhou Normal University, P.R. China}
\email{huanyinchen@aliyun.com}

\author{S. Halicioglu}
\address{Sait Halicioglu,  Department of Mathematics, Ankara University, Turkey}
\email{halici@ankara.edu.tr}

\newtheorem{thm}{Theorem}[section]
\newtheorem{lem}[thm]{Lemma}
\newtheorem{prop}[thm]{Proposition}
\newtheorem{cor}[thm]{Corollary}
\newtheorem{exs}[thm]{Examples}
\newtheorem{defn}[thm]{Definition}
\newtheorem{nota}{Notation}
\newtheorem{rem}[thm]{Remark}
\newtheorem{ex}[thm]{Example}
\newtheorem{que}[thm]{Question}

\begin{abstract} In this paper, we introduce a class of rings
which is a generalization of reflexive rings and $J$-reversible rings. Let $R$ be a ring
with identity and $J(R)$ denote the Jacobson radical of $R$. A
ring $R$ is called {\it $J$-reflexive}
 if for any $a$, $b \in R$,  $aRb = 0$ implies $bRa \subseteq J(R)$.
 We give some characterizations of a $J$-reflexive ring. We prove that some results of
 reflexive rings can be extended to
$J$-reflexive rings for this general setting.  We conclude some
 relations between $J$-reflexive rings and some related rings. We investigate some extensions of a ring
 which satisfies the $J$-reflexive property and we show that the $J$-reflexive
 property is Morita invariant.
  \vspace{2mm}


\end{abstract}

\maketitle

\section{Introduction} Throughout this paper all rings are associative
with identity unless otherwise stated. We write
$M_n(R)$ for the ring of all $n\times n$ matrices and $T_n(R)$ for
the ring of all $n\times n$ upper triangular matrices over a ring
$R$. Also we write $R[x]$, $R[[x]]$, $N(R)$, $U(R)$ and $J(R)$ for
the polynomial ring, the power series ring over a ring $R$, the
set of all nilpotent elements, the set of all invertible elements
and the Jacobson radical of a ring $R$, respectively. Also $\Bbb Z$ denote the ring of integers.

In \cite{Mas}, Mason introduced the reflexive property for ideals.
Let $R$ be a ring (without identity)
 and $I$ an ideal of $R$. Then $I$ is called \textit{reflexive}, if $aRb\subseteq I$
 for $a,b\in R$ implies $bRa\subseteq I$.  It is clear that every semiprime ideal
 is reflexive. Also, the ring $R$ is called \textit{reflexive}, if $0$ is a reflexive ideal
  (i.e $aRb=0$ implies $bRa=0$ for $a,b\in R$). In \cite{Kwaklee}, Kwak and
  Lee studied reflexive rings. They investigate reflexive property of rings
related to matrix rings and polynomial rings.  According to Cohn
\cite{Co}, a ring $R$ is said to be
  \textit{reversible} if for any
 $a$, $b\in R$, $ab = 0$ implies $ba = 0$. It is clear that every reversible
 ring is reflexive. Recently, as a generalization of a reversible ring,
 so-called $J$-reversible ring has been studied in \cite{MCSA}. A ring $R$ is called \textit{$J$-reversible},
 $ab=0$  implies that $ba\in J(R)$ for $a,b\in R$. As an application it is shown
 that every $J$-clean ring is directly finite.  Motivated by these studies,
 we introduce a class of rings which generalize $J$-reversible rings and reflexive rings. A ring $R$ is called
\textit{$J$-reflexive},
 $bRa\subseteq J(R)$ whenever $aRb=0$ for $a,b\in R$.

 We summarize the contents of this paper. In Section 2, we study main properties of
 $J$-reflexive rings.  We give some characterizations of $J$-reflexive rings.
 We prove that every $J$-reversible ring is
$J$-reflexive and we supply an example (Example \ref{exrev-refx})
to show that the
 converse is not true in general. Moreover, we see that if $R$ is
 a Baer ring, then $J$-reversible
 rings are $J$-reflexive. It is clear that reflexive rings are $J$-reflexive.
 Example \ref{mete} shows that $J$-reflexive rings need not be
 reflexive. We give a necessary and sufficient condition for a quotient ring to be
$J$-reflexive. Also we conclude some results which investigate relations between $J$-reflexive rings and some class of rings. With our finding, we prove that every uniquely
clean ring is  $J$-reflexive and quasi-duo rings are
$J$-reflexive. Moreover, we shows that the converse is not true in general.
 Being Morita invariant propety is very important for class of rings. A ring-theoretic property $\mathcal P$ is \textit{Morita invariant} if and only if whenever
 a ring $R$ satisfies $\mathcal P$ so does $eRe$, for any full idempotent $e$ and $M_n(R)$ for any $n>1$. There are a lot of studies on Morita invarivant property of rings. In Section 3, we prove that the $J$-reflexive property is Morita invariant.
 Furthermore, we study the $J$-reflexive property in several kinds of ring extensions
 (Dorroh extension, upper triangular
 matrix ring, Laurent polynomial ring, trivial extension etc.).

\section{$J$-Reflexive Rings}
In this section we define the $J$-reflexive property of a ring. We
investigate some properties of $J$-reflexive rings and exert
relations between $J$-reflexive rings and some related rings.

\begin{defn}
A ring $R$ is called $J$-reflexive, $aRb=0$ implies that
$bRa\subseteq J(R)$, for $a,b\in R$.
\end{defn}

For a nonempty subset $X$ of a ring $R$, the set $r_R(X)=\{a\in R:
Xa=0\}$ is called \textit{the right annihilator of $X$ in $R$} and
the set $l_R(X)=\{b\in R:bX=0\}$ is called \textit{the left
annihilator of $X$ in $R$}.

Now we give our main characterization for $J$-reflexive rings.

\begin{thm}
The following are equivalent for a ring $R$.
\begin{enumerate}
\item $R$ is $J$-reflexive.
\item For all $a\in R$, $r_R(aR)Ra\subseteq J(R)$ and $aRl_R(Ra)\subseteq
J(R)$.
\item $IRK=0$ implies $KRI\subseteq J(R)$ for every nonempty
subsets $I,K$ of $R$.
\item $<a><b>=0$ implies $<b><a>\subseteq J(R)$ for any $a,b\in R$.
\item $IK=0$ implies $KI\subseteq J(R)$ for every right (left)
ideals $I,K$ of $R$.
\item $IK=0$ implies $KI\subseteq J(R)$ for every ideals $I,K$ of
$R$.
\end{enumerate}
\end{thm}

\begin{proof}
$(1) \Rightarrow (2)$ Let $b\in r_R(aR)$. Then $aRb=0$ for $a,b\in
R$. Since $R$ is $J$-reflexive, $bRa\subseteq J(R)$. So we have
$r_R(Ra)Ra\subseteq J(R)$.
 Similarly, one can show that
$aRl_R(Ra)\subseteq J(R)$.\\
$(2) \Rightarrow (1)$ Assume that $aRb=0$ for $a,b\in R$. Then,
$b\in r_R(aR)$. By (2) we have $bRa\subseteq J(R)$. So $R$ is a $J$-reflexive ring. \\
$(3) \Rightarrow (4) \Rightarrow (5)\Rightarrow (6)$ It is clear.\\
$(6) \Rightarrow (1)$ Let $aRb=0$ for $a,b\in R$. Then $RaRRbR=0$
. By hypothesis, $RbRRaR\subseteq J(R)$. As $bRa\subseteq RbRRaR$,
we have $bRa\subseteq J(R)$.\\
$(1) \Rightarrow (3)$ Assume that $IRK=0$ for nonempty subsets
$I,K$ of $R$. Then for any $a\in I$ and $b\in K$, $aRb=0$. As $R$
is $J$-reflexive, $bRa\subseteq J(R)$. This implies that
$KRI\subseteq J(R)$.
\end{proof}

Examples of $J$-reflexive rings are abundant. All reduced rings,
symmetric rings, reversible rings and reflexive rings are
$J$-reflexive.  In the sequel, we show that every  $J$-reversible
ring, uniquely clean ring and every right (left) quasi-duo
 ring is  $J$-reflexive.

\begin{prop}\label{rev-refx}
Every $J$-reversible ring is $J$-reflexive.
\end{prop}

\begin{proof}
Let $R$ be a $J$-reversible ring and $aRb=0$ for some $a,b\in R$.
Then $ab=0$ and $abr=0$ for any $r\in R$. As $R$ is
$J$-reversible, $bra\in J(R)$. Hence, $bRa\subseteq J(R)$.
\end{proof}
The converse statement of Proposition \ref{rev-refx} is not true
in general as the following example shows.

\begin{ex}\label{exrev-refx}\rm{
Consider the ring $R=M_2(\Bbb Z)$. It can be easily shown that $R$
is a $J$-reflexive ring. Let $A=\left(%
\begin{array}{cc}
  0 & 0 \\
  0 & 1 \\
\end{array}%
\right)$, $B=\left(%
\begin{array}{cc}
  1 & 1 \\
  0 & 0 \\
\end{array}%
\right) \in R$. Although $AB=0$, $BA=\left(%
\begin{array}{cc}
0 & 1 \\
0 & 0 \\
\end{array}%
\right)\notin J(R)$. So $R$ is not $J$-reversible.}
\end{ex}

Recall that a ring $R$ is called \emph{Baer} if the right (left)
annihilator of every nonempty subset of $R$ is generated by an
idempotent (see for details \cite{Kap}). We show that the converse
statement of Proposition \ref{rev-refx} is true for Baer rings.

\begin{thm}
Let $R$ be a Baer ring. Then the following are equivalent.
\begin{enumerate}
\item $R$ is a $J$-reversible ring.
\item $R$ is a $J$-reflexive ring.
\end{enumerate}
\end{thm}

\begin{proof}
$(1)\Rightarrow (2)$ It is clear by Proposition \ref{rev-refx}.\\
$(2)\Rightarrow (1)$ Let $ab=0$ for $a,b\in R$. Then $abR=0$ and
so $a\in l_R(bR)$. As $R$ is a Baer ring, there exists an
idempotent $e\in R$ such that $l_R(bR)=eR$. Then we have $eRbR=0$.
Since $R$ is $J$-reflexive, $bReR\subseteq J(R)$ and so $ba\in
J(R)$, as desired.
\end{proof}

Though reflexive rings are $J$-reflexive, $J$-reflexive rings are
not reflexive as the following example shows.

\begin{ex}\label{mete}
\rm{
Let $R$ be a commutative ring. Consider the ring $$S=\left \{
\left(%
\begin{array}{ccc}
  a & b & c \\
  0 & a & d \\
  0 & 0 & a \\
\end{array}%
\right)~:~a,b,c,d\in R \right \}.$$  By \cite[Propoosition
3.7]{MCSA}, $S$ is $J$-reversible and by Proposition \ref{rev-refx}, it is $J$-reflexive.
For $A=\left(%
\begin{array}{ccc}
  0 & 0 & 0 \\
  0 & 0 & 1 \\
  0 & 0 & 0 \\
\end{array}%
\right)$, $B=\left(%
\begin{array}{ccc}
  0 & 1 & 0 \\
  0 & 0 & 0 \\
  0 & 0 & 0 \\
\end{array}%
\right) \in S$, $ASB=0$ but $BSA\neq 0$. Thus, $S$ is not a
reflexive ring.}
\end{ex}

The following result can be easily obtain by the definition of
$J$-reflexive rings.

\begin{cor}\label{r/jrcomm} The following are equivalent for a ring $R$.
\begin{enumerate}
\item If $R/J(R)$ is reflexive, then $R$ is $J$-reflexive.
\item If $R/J(R)$ is commutative, then $R$ is $J$-reflexive.
\end{enumerate}
\end{cor}

An element $a$ in a ring $R$ is called \textit{uniquely clean} if
$a = e + u$ where $e^ 2 = e\in R$ and $u\in U(R)$ and this
representation is unique. A ring $R$ is called a \textit{uniquely
clean ring} if every element of $R$ is uniquely clean (see
\cite{NZ}).
\begin{cor}\label{uniclean}
    Every uniquely clean ring is $J$-reflexive.
\end{cor}

\begin{proof}
Assume that $R$ is uniquely clean. Then $R/J(R)$ is Boolean by \cite[Theorem 20]{NZ}.
Hence, $R$ is $J$-reflexive by Corollary \ref{r/jrcomm}.
\end{proof}

The converse statement of Corollary \ref{uniclean} is not true in
general as the following example shows.

\begin{ex}\rm{ For a commutative ring $R$, consider the ring $M_2(R)$.
Since $M_2(R)$ is not an
    abelian ring, $M_2(R)$ is not a uniquely clean ring. Also, it can be easily shown that
    $M_2(R)$ is a $J$-reflexive ring by Theorem \ref{morita_inv.}.}
\end{ex}

\begin{prop}
Let $R$ be a ring. If $N(R)\subseteq J(R)$, then $R$ is $J$-reflexive.
\end{prop}

\begin{proof}
Assume that $aRb=0$ for $a,b\in R$. Then for any $r\in R$, $arb=0$ and so $ab=0$. Hence,
$(bra)^2=brabra=0$ for all $r\in R$. So $bra\in N(R)$. By hypothesis we have $bra\in J(R)$,
 as asserted.
\end{proof}
A ring $R$ is called \textit{right (left) quasi-duo} if every
right (left) maximal ideal of $R$ is
 an ideal \cite{Laam}.
\begin{cor}
Every right (left) quasi-duo ring is $J$-reflexive.
\end{cor}

\begin{proof}
It is clear by \cite[Lemma 2.3]{Yu}.
\end{proof}

We now give a necessary and sufficient condition for a quotient
ring to be $J$-reflexive.

\begin{thm}\label{quotient}
Let $R$ be a ring and $I$  a nilpotent ideal of $R$. Then $R$ is
$J$-reflexive if and only if $R/I$ is $J$-reflexive.
\end{thm}

\begin{proof} Let $R/I=\overline{R}$, $a+I=\overline{a}\in \overline{R}$ and
 $\overline{aRb}=\overline{0}$ for $\overline {a},\overline{ b}\in \overline {R}$.
 So $aRb\subseteq I$. As $I$ is nilpotent there exists $k\in \Bbb Z^+$ such that
 $(RaRbR)^k=0$. $(RbRaR)^k\subseteq J(R)$, since $R$ is $J$-reflexive. Thus
  $RbRaR\subseteq J(R)$ as Jacobson radical is semiprime. Hence
  $\overline{RbRaR}\subseteq J(R)/I=J(\overline{R})$. So $\overline{bRa}\subseteq
  J(\overline{R})$.\\
Conversely, assume that $aRb=0$ for $a,b\in R$. Then
$\overline{aRb}=\overline{0}$.
 So $aRb\subseteq I$ and $RaRbR\subseteq I$. Therefore there exists $k\in \Bbb Z^+$ such
 that $I^k=0$, and so $(RaRbR)^k=RaRbRaRbR\cdots RaRbR=0$. Hence,
 $\overline{(RaRbR)^k}=\overline{0}$. Since $R/I$ is $J$-reflexive,
 $\overline{(RbRaRbRaR\cdots RbRaR)}=\overline{(RbRaR)}^k=\subseteq J(\overline{R})$.
 As Jacobson radical is semiprime ideal, we have $RbRaR\subseteq J(\overline{R})$.
 Thus, $\overline{bRa}\subseteq J(\overline{R})$. Hence, for all $\overline r\in
\overline R$, we have $\overline 1-(\overline{bra})\overline x\in
U(\overline R)$ for some $\overline x\in J(\overline{R})$. Then,
there exists $\overline s\in \overline R$ such that $(\overline
1-(\overline{bra})\overline x)\overline s=\overline 1$. Hence,
$1-(1-bra)xs\in I$. As every nilpotent ideal is nil, $1-(brax)s\in
U(R)$. This implies that $bra\in J(R)$ and so $bRa\subseteq J(R)$,
as desired.
\end{proof}

\begin{cor}\label{quot2}
   Let $R$ be a ring. Then the following are satisfied.
    \begin{enumerate}
        \item If $J(R)$ is a nilpotent ideal, then $R$ is $J$-reflexive if and only if $R/J(R)$ is $J$-reflexive.
        \item If $R$ is an Artinian ring, then $R$ is $J$-reflexive if and only if $R/J(R)$ is $J$-reflexive.
    \end{enumerate}
\end{cor}

\begin{proof}
$(1)$ It is clear by Theorem \ref{quotient}.\\
$(2)$ Since Jacobson radical of Artinian ring is nilpotent,
    it is clear by (1).
\end{proof}

\begin{prop}
Let $R$ be a ring and $I$ an ideal of $R$ with $I\subseteq J(R)$.
If $R/I$ is $J$-reflexive, then $R$ is $J$-reflexive.
\end{prop}

\begin{proof}
Let $\overline R=R/I$ and for $\overline a=a+I\in R/I$.
Assume that $aRb=0$ for $a,b\in R$. So $\overline{aRb}=\overline 0
$. Since $\overline R$ is $J$-reflexive, $\overline{bRa}\subseteq
J(\overline R)$ and $\overline{bra}\in J(\overline R)$ for any
$\overline r\in \overline R$. Thus, for all $\overline x\in
\overline R$ we have $\overline 1-(\overline{bra})\overline x\in
U(\overline R)$. Then, there exists $\overline s\in \overline R$
such that $(\overline 1-(\overline{bra})\overline x)\overline
s=\overline 1$. Hence, $1-(1-brax)s\in I$. As $I$ is contained $J(R)$,
$(1-brax)s\in U(R)$. This implies that $bra\in J(R)$ and so
$bRa\subseteq J(R)$, as desired.
\end{proof}

\begin{prop}
Let $R$ be a ring and $I$ a reflexive ideal of $R$. Then $R/I$ is
$J$-reflexive.
\end{prop}

\begin{proof}
Let $\overline R=R/I$ and for $\overline a=a+I\in R/I$. Suppose
that $\overline {aRb}=\overline 0$ for $\overline{a},\overline{b}\in \overline{R}$.
Then $aRb\subseteq I$. Since
$I$ is a reflexive ideal, we have $bRa\subseteq I$. Hence,
$\overline{bRa}=\overline 0\in J(\overline R)$.
\end{proof}

\begin{thm}\label{subdirect}
Every subdirect product of $J$-reflexive ring is $J$-reflexive.
\end{thm}

\begin{proof}
Let $R$ be a ring, $I,K$ ideals of $R$ and $R$ a subdirect product
of $R/I$ and $R/K$. Assume that $R/I$ and $R/K$ are $J$-reflexive.
Let $aRb=0$ for $a,b\in R$. Then $\overline{aRb}=\overline{0}$ in
$R/I$ and $R/K$. Since $R/I$ and $R/K$ are $J$-reflexive,
$\overline{bRa}\subseteq J(R/I)$ and $\overline{bRa}\subseteq
J(R/K)$. Then for each $x\in R$ we have $\overline{1-brax}\in
U(R/I)$ and $\overline{1-brax}\in U(R/K)$.
 Hence, there exist $y\in R/I$ and $z\in R/K$ such that
 $\overline{(1-brax)y}=\overline{1}\in R/I$ and $\overline{(1-brax)z}=\overline{1}\in R/K$.
 So $1-(1-brax)y\in I$ and $1-(1-brax)z\in K$. If we multiply the last two elements,
 we have $(1-(1-brax)y)(1-(1-brax)z)\in IK\subseteq I\cap K=0$. Thus, $1-(1-brax)t=0$
 and $(1-brax)t=1$. This implies that $bRa\subseteq J(R)$.
\end{proof}

\begin{cor}
Let $I$ and $K$ be ideals of a ring $R$. If $R/I$ and $R/K$ are $J$-reflexive,
then $R/I\cap K$ is $J$-reflexive.
\end{cor}

\begin{proof}
Let $\alpha:R/(I\cap K)\rightarrow R/I$ and $\beta:R/(I\cap K)\rightarrow R/K$
where $\alpha(r+(I\cap K))=r+I$ and $\beta(r+(I\cap K))=r+K$. It can be shown
that $\alpha$ and $\beta$ surjective ring homomorphisms and ker$\alpha\cap$ker$\beta=0$.
Hence $R/(I\cap K)$ is subdirect product of $R/I$ and $R/K$. Therefore, $R/(I\cap K)$ by
Theorem \ref{subdirect}.
\end{proof}

\begin{cor}
Let $R$ be a ring and $I,K$ of ideals of $R$. If $R/I$ and $R/K$ are $J$-reflexive,
then $R/IK$ is $J$-reflexive.
\end{cor}

\begin{proof}
Assume that $R/I$ and $R/K$ are $J$-reflexive. Since $$R/I\cap K
\cong (R/IK)/(I\cap K/ IK)$$ and $(I\cap K/IK)^2=0$, we complete
the proof by Theorem \ref{quotient}.
\end{proof}

\section{Extensions of $J$-reflexive rings}

In this section we show that several extensions (Dorroh extension,
upper triangular matrix ring, Laurent polynomial ring, trivial
extension etc.) of a $J$-reflexive ring are $J$-reflexive. In
particular, it is proved that the $J$-reflexive condition is
Morita invariant.

 Two rings $R$ and $S$ are said to be \textit{Morita equivalent} if the categories of
 all right $R$-modules and  all right $S$-modules are equivalent. Properties
  shared between equivalent rings are called \textit{Morita invariant properties}.
  $\mathcal{P}$ is Morita invariant if and only if whenever a ring $R$ satisfies
  $\mathcal P$, then so does $eRe$ for every full idempotent $e$ and so does every
  matrix ring $M_n(R)$ for every positive integer $n$.

 Next result shows that the property of $J$-reflexivitiy is Morita invariant.

\begin{thm}\label{morita_inv.} Let $R$ be a ring. Then we have the following.
\begin{enumerate}
\item If $R$ is $J$-reflexive, then $eRe$ is $J$-reflexive for all
idempotent $e\in R$.
\item $R$ is a $J$-reflexive ring if and only if $M_n(R)$ is
$J$-reflexive for any positive integer $n$.
\end{enumerate}
\end{thm}

\begin{proof}
$(1)$ Assume that $R$ is a $J$-reflexive ring. Let $a,b\in eRe$
with $aeReb=0$. As $R$ is $J$-reflexive, $ebRae=ebRea\subseteq
J(eRe)=eJ(R)e$. This implies that $eRe$ is a $J$-reflexive ring.\\
$(2)$ Assume that $M_n(R)$ is $J$-reflexive ring. It is clear that
$R$ is $J$-reflexive by $(1)$. Conversely, suppose that $R$ is
$J$-reflexive and $I,K$ are an ideals of $M_n(R)$ such that
$IK=0$. Then, there exist $I_1,K_1$ ideals of $R$ such that
$I=M_n(I_1)$ and $K=M_n(K_1)$. So
$0=IK=M_n(I_1)M_n(K_1)=M_n(I_1K_1)$. Thus, $I_1K_1=0$. Since $R$
is $J$-reflexive, $K_1I_1\subseteq J(R)$. This implies that
$KI=M_n(K_1)M_n(I_1)=M_n(K_1I_1)\subseteq J(M_n(R))=M_n(J(R))$. This
completes the proof.
\end{proof}

\begin{cor}
    Let $M$ be a finitely generated projective modules over a $J$-reflexive ring $R$.
    Then End$_R(M)$ is $J$-reflexive.
\end{cor}

\begin{proof}
    It is obvious by Theorem \ref{morita_inv.}.
\end{proof}

\begin{prop}
The following are equivalent for a ring $R$.
\begin{enumerate}
\item $R$ is $J$-reflexive.
\item $M=\left\{\left(%
\begin{array}{cccc}
    r &x_{12}&\cdots &x_{1n}  \\
    0 &r&\cdots &x_{2n} \\
    \vdots & \vdots&\vdots&\vdots \\
    0 &\cdots&\cdots&r \\
\end{array}%
\right):r\in R,~x_{ij}\in R\right\}$ is $J$-reflexive.
\end{enumerate}
\end{prop}
\begin{proof}
$(1)\Leftrightarrow (2)$ Take $I=\left(%
\begin{array}{cccc}
0&x_{12}&\cdots &x_{1n}\\
0 &0&\cdots &x_{2n}\\
\vdots & \vdots&\vdots&\vdots \\
0 &\cdots&\cdots&0\\
\end{array}%
\right)$. The proof is clear by Theorem \ref{quotient}.
\end{proof}

Recall that the \textit{trivial extension} of $R$ by an $R$-module
$M$ is the ring denoted by $R \propto M$ whose underlying additive
group is $R\oplus M$ with multiplication given by $(r, m)(r ' , m'
) = (rr' , rm'+mr' )$. The ring $R\propto M$ is isomorphic to
$S=\left\{\left(%
\begin{array}{cc}
  x & y \\
  0 & x \\
\end{array}%
\right) :x\in R,~y\in M\right\}$ under the usual matrix
operations.

\begin{prop} The following are equivalent for a ring $R$.
\begin{enumerate}
\item The trivial extension $R\propto R$ of the ring $R$ is
$J$-reflexive.
\item $R$ is a $J$-reflexive ring.
\end{enumerate}
\end{prop}

\begin{proof}
$(1)\Rightarrow (2)$ Assume that $R\propto R$ is $J$-reflexive.
Let $aRb=0$ for $a,b\in R$. Then, for $A=\left(%
\begin{array}{cc}
  a & 0 \\
  0 & a \\
\end{array}%
\right), ~B=\left(%
\begin{array}{cc}
  b & 0 \\
  0 & b \\
\end{array}%
\right)\in R\propto R$, we have $A(R\propto R)B=\left(%
\begin{array}{cc}
  aRb & aRb \\
  0 & aRb \\
\end{array}%
\right)=\left(%
\begin{array}{cc}
  0 & 0 \\
  0 & 0 \\
\end{array}%
\right)$. As $R\propto R$ is $J$-reflexive, $B(R\propto
R)A\subseteq J(R\propto R)$. Hence, $bRa\subseteq J(R)$.\\
$(2)\Rightarrow (1)$ Suppose that $R$ is $J$-reflexive. Let
$A(R\propto R)B=0$ for $A=\left(%
\begin{array}{cc}
  a & x \\
  0 & a \\
\end{array}%
\right), ~B=\left(%
\begin{array}{cc}
  b & y \\
  0 & b \\
\end{array}%
\right)\in R\propto R$. Then for any $M=\left(%
\begin{array}{cc}
  s & t \\
  0 & s \\
\end{array}%
\right)\in R\propto R$, we have $AMB=\left(%
\begin{array}{cc}
  asb & asy+atb+xsb \\
  0 & asb \\
\end{array}%
\right)=\left(%
\begin{array}{cc}
  0 & 0 \\
  0 & 0 \\
\end{array}%
\right)$. Since $R$ is $J$-reflexive and $aRb=0$, we conclude that
$bsa\in J(R)$ for any $s\in R$. Note that $J(R\propto R)=\left(%
\begin{array}{cc}
  J(R) & R \\
  0 & J(R) \\
\end{array}%
\right)$. Hence, $B(R\propto R)A\subseteq J(R\propto R)$, as
asserted.
\end{proof}

\begin{prop}\label{product}
Let $\{R_i\}_{i\in \mathcal I}$ indexed set of the ring $R_i$.
Then $R_i$ is $J$-reflexive for all $i\in \mathcal I$ if and only
if $\prod_{i\in \mathcal I}R_i$ is $J$-reflexive.
\end{prop}

\begin{proof}
$\Rightarrow$ Let $\prod_{i\in \mathcal I}M_iK_i=0$ for ideals
$\prod_{i\in \mathcal I}M_i,\prod_{i\in \mathcal I}K_i$ of
$\prod_{i\in \mathcal I}R_i$. Then $\prod_{i\in \mathcal
I}M_iK_i=0$. Therefore, $M_iK_i=0$ for all $i\in \mathcal I$.
Since $R_i$ is $J$-reflexive, $K_iM_i\subseteq J(R_i)$ for all
$i\in \mathcal I$. So $\prod_{i\in \mathcal I}K_i\prod_{i\in
\mathcal I}M_i=\prod_{i\in \mathcal I}K_iM_i\subseteq
J(\prod_{i\in \mathcal I}R_i)=\prod_{i\in \mathcal I}J(R_i)$.\\
$\Leftarrow$ Assume that $M_{\phi} K_{\phi}=0$ for ideals
$M_{\phi}, K_{\phi}$ of $R_{\phi}$. Choose $M=(M_{\phi})_{\phi\in
\mathcal I}$ and $K=(K_{\phi})_{\phi\in \mathcal I}$ as only
$\phi$ components are nonzero ideal. So $M$ and $K$ are ideals of
$\prod_{i\in \mathcal I}R_i$. Also we have $MK=0$. As $\prod_{i\in
\mathcal I}R_i$ is $J$-reflexive, $KM\subseteq J(\prod_{i\in
\mathcal I}R_i)$. Thus, $K_{\phi}M_{\phi}\subseteq J(R_{\phi})$.
\end{proof}

\begin{prop}
The following are equivalent for a ring $R$.
\begin{enumerate}
\item $R$ is a $J$-reflexive ring.
\item $T_n(R)$ is $J$-reflexive for any $n\in \Bbb Z^+$.
\end{enumerate}
\end{prop}

\begin{proof}
$(1)\Rightarrow (2)$ For $n=1$ it is clear. Consider the ring $T_2(R)$.
Choose the ideal $I=\left(%
\begin{array}{cc}
0 & R \\
0 & 0 \\
\end{array}%
\right)$. It is clear that $I^2=0$. So $T_2(R)/I\cong R\times R$.
By Proposition \ref{product}, $T_2(R)/I$ is $J$-reflexive. Hence
$T_2(R)$ is $J$-reflexive, by Theorem \ref{quotient}. By
induction,  $T_n(R)$ is $J$-reflexive for
any $n\in \Bbb Z^+$.\\
$(2)\Rightarrow (1)$ It is evident from Theorem
\ref{morita_inv.}(1).
\end{proof}

\begin{prop}
Let $R$ be a ring and $e^2=e\in R$ is central. Then, $R$ is a
$J$-reflexive ring if and only if $eR$ and $(1-e)R$ are
$J$-reflexive.
\end{prop}

\begin{proof} The necessity is obvious by Theorem \ref{morita_inv.}.
For the sufficiency suppose that $eR$ and $(1-e)R$ are
$J$-reflexive for a central idempotent $e\in R$. It is well-known
that $R\cong eR\times(1-e)R$. By Proposition \ref{product}, $R$ is
$J$-reflexive.
\end{proof}

For an algebra $R$ over a commutative ring $S$, the \emph{ Dorroh
extension} $I(R; S )$ of $R$ by $S$ is the additive abelian group
$I(R; S) = R\oplus S$ with multiplication $(r, v)(s,w) = (rs, rw +
vs + vw)$.

\begin{prop}
Let $R$ be a ring and $M=I(R;S)$ Dorroh extension of $R$ by
commutative ring $S$. Assume that for all $s\in S$ there exists
$s'\in S$ such that $s+s'+ss' = 0$. Then the following are
equivalent.
\begin{enumerate}
\item $R$ is $J$-reflexive.
\item $M$ is $J$-reflexive.
\end{enumerate}
\end{prop}

\begin{proof}
$(1) \Rightarrow (2) $ Let $(a_1,b_1)M(a_2,b_2)=(0,0)$ for
$(a_1,b_1),(a_2,b_2)\in M$. So for any $(x,y)\in M$, we have
$(a_1,b_1)(x,y)(a_2,b_2)=(0,0)$. Then
$(a_1xa_2,a_1xb_2+a_1ya_2+b_1xa_2+b_1ya_2+a_1yb_2+b_1xb_2+b_1yb_2)=(0,0)$.
Hence, $a_1xa_2=0$ and
$a_1xb_2+a_1ya_2+b_1xa_2+b_1ya_2+a_1yb_2+b_1xb_2+b_1yb_2=0$. As
$R$ is $J$-reflexive, $a_2xa_1\in J(R)$ for any $x\in R$. Thus,
$(a_2,b_2)(x,y)(a_1,b_1)=(a_2xa_1,*)$. By hypothesis,
$(0,S)\subseteq J(M)$. It can be easy to show that $(a_2xa_1,0)\in
J(M)$ for each $x\in R$. Therefore, $(a_2,b_2)S(a_1,b_1)\subseteq
J(M)$.\\
$(2) \Rightarrow (1) $ Let $aRb=0$ for $a,b\in R$. Then
$(a,0)M(b,0)=(0,0)$. Since $M$ is $J$-reflexive,
$(b,0)M(a,0)\subseteq J(M)$. By hypothesis, $(0,S)\subseteq J(M)$.
This implies that $(bRa,0)\subseteq J(S)$. Hence, $bRa\subseteq
J(R)$.
\end{proof}

If $R$ is a ring and $f: R \rightarrow R$ is a ring homomorphism,
let $R[[x, f]]$ denote \emph{the ring of skew formal power series
over $R$}; that is all formal power series in $x$ with
coefficients from $R$ with multiplication defined by $xr = f(r)x$
for all $r \in R$. Note that $J(R[[x, f]]) = J(R)+ < x>$. Since
$R[[x, f]] \cong I(R;< x>)$ where $< x >$ is the ideal generated
by $x$, we have the following result.

\begin{cor}
Let R be a ring and $ f: R\rightarrow R$ a ring homomorphism. Then
the following are equivalent.
\begin{enumerate}
\item $R$ is a $J$-reflexive ring.
\item $R[[x,f]]$ is $J$-reflexive.
\end{enumerate}
\end{cor}

If $f$ is taken $f=1_R:R\rightarrow R$ $(i.e.~1_R(r)=r ~\mbox{for
all}~ r\in R)$, we have $R[[x]] = R[[x, 1_R]]$ is the ring of
formal power series over $R$.
\begin{cor}
The following are equivalent for a ring $R$.
\begin{enumerate}
\item $R$ is a $J$-reflexive ring.
\item $R[[x]]$ is $J$-reflexive.
\end{enumerate}
\end{cor}

Let $R$ be a ring and $u\in R$. Recall that $u$ is \emph{right
regular} if $ur=0$ implies $r=0$ for $r\in R$. Similarly,
\emph{left regular} element can be defined. An element is
\emph{regular} if it is both left and right regular.

\begin{prop}\label{localization}
Let $R$ be a ring and $M$ multiplicatively closed subset of $R$
consisting of central regular elements. Then the following are
equivalent.
\begin{enumerate}
\item $R$ is $J$-reflexive.
\item $S=M^{-1}R=\{\frac{a}{b}:a\in R,b\in M\}$ is $J$-reflexive.
\end{enumerate}
\end{prop}

\begin{proof}
$(1)\Rightarrow (2)$ Let $aSb=0$ for $a,b\in S$. So there exist
$a_1,b_1\in R$ and $u^{-1},v^{-1}\in M$ such that $a=a_1u^{-1}$
and $b=b_1v^{-1}$. Then $0=aSb=a_1u^{-1}Sb_1v^{-1}=a_1Sbv^{-1}$.
Hence for any $rs^{-1}\in S$ we have $a_1rs^{-1}bv^{-1}$. Thus,
$a_1rb_1=0$ for each $r\in R$. As $R$ is $J$-reflexive,
$b_1ra_1\in J(R)$. This implies that $b_1v^{-1}rs^{-1}a_1u^{-1}\in
J(R)$. As $J(R)\subseteq J(S)$, $aSb\subseteq J(S)$. \\
$(2) \Rightarrow (1)$ Let $aRb=0$ for $a,b\in R$ and $u,v\in M$.
So we have $auRbv=0$. Then for any $m\in M$ and $r\in R$
$aurmbv=0$. Since $S$ is $J$-reflexive, $bvrmau\in J(S)$. If we
multiply $bvrmau$ with inverses of $u,m,v$, then we have $bra\in
J(R)$ for any $r\in R$. This completes the proof.
\end{proof}

The following result is a direct consequence of Proposition
\ref{localization}.

\begin{cor}\label{polynomial}
Let $R$ be a ring. Then the following are equivalent.
\begin{enumerate}
\item $R[x]$ is $J$-reflexive.
\item $R[x,x^{-1}]$ is $J$-reflexive.
\end{enumerate}
\end{cor}

\end{document}